\documentclass[12pt, a4paper]{amsart}
\usepackage{amsmath}
\usepackage{geometry,amsthm,graphics,tabularx,amssymb,shapepar}
\usepackage{amscd}
\usepackage[all]{xypic}

\newcommand{\ft}{\mathbf t}

\newcommand{\fs}{\mathbf s}

\newcommand{\BR}{{\mathbb {R}}}

\newcommand{\RG}{{\mathrm {G}}}

\newcommand{\RK}{{\mathrm {K}}}

\newcommand{\RN}{{\mathrm {N}}}

\newcommand{\RP}{{\mathrm {P}}}

\newcommand{\RR}{{\mathrm {R}}}

\newcommand{\RT}{{\mathrm {T}}}
\newcommand{\RU}{{\mathrm {U}}}

\newcommand{\GL}{{\mathrm{GL}}}

\newcommand{\ind}{{\mathrm{ind}}}

\newcommand{\rk}{{\mathrm{k}}}

\newcommand{\con}{\textit{C}}

\newcommand{\diag}{\operatorname{diag}}

\newcommand{\oN}{\operatorname{N}}

\newcommand{\oZ}{\operatorname{Z}}

\renewcommand{\rk}{\mathrm k}

\newcommand{\C}{\mathbb{C}}
\newcommand{\R}{\mathbb R}

\newcommand{\abs}[1]{\lvert#1\rvert}

\newcommand{\rd}{\mathrm{d}}
\newcommand{\wo}{\widehat{\otimes}}
\newcommand{\la}{\langle}
\newcommand{\ra}{\rangle}

\newcommand{\be}{\begin {equation}}
\newcommand{\ee}{\end {equation}}
\newcommand{\bee}{\begin {equation*}}
\newcommand{\eee}{\end {equation*}}

\theoremstyle{Theorem}

\theoremstyle{Theorem}
\newtheorem{introconjecture}{Conjecture}

\newtheorem{introtheorem}[introconjecture]{Theorem}

\theoremstyle{Theorem}
\newtheorem{lem}{Lemma}[section]

\theoremstyle{Theorem}
\newtheorem{prp}{Proposition}[section]

\newtheorem{lemp}[prp]{Lemma}

\theoremstyle{Plain}

\theoremstyle{Definition}

\title{Uniqueness of Rankin-Selberg periods}

\author{Fulin Chen}
\address{Academy of Mathematics and Systems Science, Chinese Academy of Sciences,
Beijing, 100190, China} \email{fulinchen@amss.ac.cn }

\author{Binyong Sun}
\address{Academy of Mathematics and Systems Science\\
Chinese Academy of Sciences\\
Beijing, 100190, P.R. China} \email{sun@math.ac.cn}


\begin{document}

\begin{abstract}Let $\rk$ be a local field of characteristic zero. Rankin-Selberg's local zeta integrals
produce linear functionals on generic irreducible admissible smooth
representations of $\GL_n(\rk)\times \GL_r(\rk)$, with certain
invariance properties. We show that  up to scalar multiplication,
these linear functionals are determined by the invariance
properties.
\end{abstract}
\maketitle
\section{Introduction}

The goal of this paper is to prove uniqueness of certain local
linear functionals that naturally occur in the study of
Rankin-Selberg convolutions. We first treat the archimedean case,
and leave the non-archimedean case to Section \ref{sec5}. So assume
that  $\rk$ is an  archimedean local filed throughout this paper
except for  Section \ref{sec5}.

The representations that we will consider are so called
Casselman-Wallach representations. They appear as archimedean
components of automorphic representations. Recall that a
representation of a real reductive group is called a
Casselman-Wallach representation if it is smooth, Fr\'{e}chet,  of
moderate growth, and its Harish-Chandra module has finite length.
The reader may consult \cite{C}, \cite[Chapter 11]{W} or \cite{BK}
for details about Casselman-Wallach representations.

To explain the motivation of our paper, we first give a brief
introduction to the local Rankin-Selberg integrals.  The reader is
referred to \cite{CPS,J,JPS,JS2} for more details. Fix two integers
$n>r\geq 0$. Let $\pi$ and $\sigma$ be generic irreducible
Casselman-Wallach representations of $\GL_n(\rk)$ and $\GL_r(\rk)$,
respectively.
 For abbreviation and as usual, we do not distinguish a representation with its underlying vector space.
For every integer $m\geq 0$, write $\oN_m(\rk)$ for the subgroup of
$\GL_m(\rk)$ consisting of upper-triangular unipotent matrices.

Fix generators
\[
  \lambda_n\in \mathrm{Hom}_{\oN_n(\rk)}(\pi,\psi_n)\quad\textrm{and}\quad \lambda_r\in \mathrm{Hom}_{\oN_r(\rk)}(\sigma,\psi_r^{-1})
\]
of one dimensional spaces, where $\psi_n$ is a generic unitary
character of $\oN_n(\rk)$, and $\psi_r$ is its restriction to
$\oN_r(\rk)$ through the embedding
\[
g\mapsto \left[
\begin{array}{cc}
                                                 g & 0 \\
                                                 0 & 1_{n-r} \\
                                               \end{array}
                                             \right] \ \textrm{(for every $m\geq 1$, $1_m$ denotes the identity matrix of size $m$).}
\]
 For every $u\in \pi$ and $v\in \sigma$, define the Whittaker functions
$W_u:\GL_n(\rk)\rightarrow \mathbb C$ and
$W'_v:\GL_r(\rk)\rightarrow \mathbb C$ by
$$
 W_u(g):=\lambda_n(g.u)\quad \text{ and }\quad W'_v(g):=\lambda_r(g.v).
$$
The Rankin-Selberg integral associated to the pair $(W_u,W_v')$ is
defined to be
\begin{equation}
\label{rsintegral}\Psi(s,W_u,W'_v):=\int_{\oN_r(\rk)\backslash
\GL_r(\rk)}W_u\left( \begin{bmatrix} g & 0\\ 0&
1_{n-r}\end{bmatrix}\right) W'_v(g)
|\mathrm{det}g|^{s-\frac{n-r}{2}}\mathrm{d}g,\end{equation} where
$s\in \mathbb{C}$ and $\mathrm{d}g$ denotes a right
$\GL_r(\rk)$-invariant positive Borel measure on
$\oN_r(\rk)\backslash \GL_r(\rk)$. The integral is absolutely
convergent when the real part of $s$ is sufficiently large, and has
a meromorphic continuation to the whole complex plane $\mathbb{C}$.
Moreover, for every $s_0\in \mathbb{C}$, we have a well-defined
non-zero continuous linear functional
\begin{equation}\label{psis0}
  \Psi_{s_0}: \pi \widehat{\otimes} \sigma\rightarrow \mathbb{C}, \quad u\otimes v\mapsto \frac{\Psi(s,W_u,W'_v)}{\mathrm L(s,\pi\times \sigma)}|_{s=s_0},
\end{equation}
 where $\mathrm L(s,\pi\times \sigma)$ denotes the local L-function, and the symbol ``$\widehat \otimes$" stands for the completed projective tensor product.

Define the $r$-th Rankin-Selberg subgroup $\mathrm R_r$  of
$\GL_n(\rk)$ to be \be\label{gp} \mathrm R_r:=\left\{\begin{bmatrix} g & u \\
0 & h \end{bmatrix}\mid g\in \GL_r(\rk), u\in \mathrm
M^0_{r,n-r},h\in \oN_{n-r}(\rk)\right\}, \ee where
$$
\mathrm M^0_{r,n-r}:=\{x\in \mathrm M_{r,n-r}(\rk)\mid  \textrm{the
first collum of $x$ is  zero}\}.
$$
Here and henceforth, for every pair $i, j$ of non-negative integers,
$\mathrm M_{i,j}(\rk)$ denotes the space of $i\times j$-matrices
with coefficients in $\rk$.

View $\sigma$ as a representation of $\mathrm R_r$ by the inflation
through the obvious homomorphism $\mathrm R_r\rightarrow
\GL_r(\rk)$. Then the tensor product $\pi \widehat{\otimes} \sigma$
is also a representation of $\mathrm R_r$. It is easily checked that
the functional $\Psi_{s_0}$ of \eqref{psis0} has the following
invariance property:
\begin{equation}\label{equipsi}
 \Psi_{s_0}(g.w)=\chi_{s_0}(g)\Psi_{s_0}(w),\quad g\in \mathrm R_r,\,w\in \pi\widehat \otimes \sigma,
 \end{equation}
 where $\chi_{s_0}$ denotes the following character on $\mathrm R_r$:
\[
    \begin{bmatrix} g & u \\ 0 & h \end{bmatrix}\mapsto
   |\mathrm{det}g|^{\frac{n-r}{2}-s_0}\,\psi_n\left(\begin{bmatrix} 1 & 0 \\ 0 & h \end{bmatrix}\right).
\]

Note that every character on $\mathrm R_r$ factors through the
quotient map $\mathrm R_r\rightarrow \GL_r(\rk)\times
\oN_{n-r}(\rk)$. We say that a character on $\mathrm R_r$ is generic
if its descent to $\GL_r(\rk)\times \oN_{n-r}(\rk)$ has the form
$\chi_r\otimes \psi$, where $\chi_r$ is a character on $\GL_r(\rk)$,
and $\psi$ is a generic unitary character on $\oN_{n-r}(\rk)$. It is
clear that $\chi_{s_0}$ is a generic character on $\mathrm R_r$.

The following uniqueness theorem implies that the Rankin-Selberg
integral is characterized by the invariance property
\eqref{equipsi}.

\begin{introtheorem}\label{main}
For all irreducible Casselman-Wallach representations $\pi$ of
$\GL_n(\rk)$, and $\sigma$ of $\GL_r(\rk)$, and for all generic
characters $\chi$ of $\mathrm R_r$, one has that
\begin{equation}\label{homnr}
  \mathrm{dim}\ \mathrm{Hom}_{\mathrm R_r}(\pi\widehat \otimes \sigma,\chi)\le 1.
\end{equation}

\end{introtheorem}

Here $\pi\widehat \otimes \sigma$ is viewed as a representation of
$\mathrm R_r$ as before.  Theorem \ref{main} is well known in two
extremal cases: if $r=0$, then $\mathrm R_r=\oN_n(\rk)$, and Theorem
\ref{main} asserts the uniqueness of Whittaker models for
$\GL_n(\rk)$ (see \cite{Sh,CHM}); if $r=n-1$, then $\mathrm
R_r=\GL_{n-1}(\rk)$, and Theorem \ref{main} asserts the multiplicity
one theorem for $\GL_n(\rk)$ (see \cite{AG,AGRS,SZ}).

View $\rk^n$ as the space of column vectors. Write $\omega$ for the
Fr\'{e}chet space of $\C$-valued Schwartz functions on $\rk^n$. It
carries a smooth representation of $\mathrm \GL_n(\rk)$:
\be\label{action}(g.\phi)(v):=\phi(g^{-1}v),\quad g\in
\GL_n(\rk),\,\phi\in \omega,\, v\in \rk^n.\ee
 We shall prove Theorem
\ref{main} by reducing it to the following multiplicity one theorem:
\begin{introtheorem}\label{main2} (cf. \cite[Theorem C]{SZ})
For all irreducible Casselman-Wallach representations of $\pi$ and
$\pi'$ of $\GL_n(\rk)$, and all characters $\chi'$ on $\GL_n(\rk)$,
one has that
\begin{equation}\label{homnn}
  \mathrm{dim}\ \mathrm{Hom}_{\GL_n(\rk)}(\pi\widehat \otimes \pi'\widehat \otimes  \omega,\chi')\le 1.
\end{equation}

\end{introtheorem}

We remark that Rankin-Selberg method implies that the hom spaces of
\eqref{homnr} and \eqref{homnn} are at least one dimensional when
the representations involved are generic.

The authors thank Dihua Jiang for helpful discussions. The second
named author is supported by NSFC Grant 11222101.

\section{Proof of Theorem \ref{main}}\label{sec2}

 For simplicity, write $\mathrm G_m:=\GL_m(\rk)$ and $\RN_m:=\RN_m(\rk)$ for every non-negative integer $m$. Put
\be\label{pb}
  \mathrm P_{r,n-r}:=\left\{\begin{bmatrix} g & u \\
0 & h \end{bmatrix} \mid  g\in \mathrm G_r, \,h\in \mathrm
G_{n-r},\, u\in \mathrm M_{r,n-r}(\rk)\right\}.
\ee
Fix a generic irreducible Casselman-Wallach representation $\tau$
of $\mathrm G_{n-r}$. For every $s\in \mathbb C$, denote by $\tau_s$
the representation of $\mathrm G_{n-r}$ which has the same
underlying space as that of $\tau$, and has the action
$$\tau_s(g)=|\det(g)|_\rk^{-s} \tau(g),\qquad g\in \mathrm G_{n-r},$$
where ``$\,\abs{\,\cdot\,}_\rk$" denotes the normalized absolute value on
$\rk$.

Form the un-normalized smooth induction
\begin{eqnarray*}
% \nonumber to remove numbering (before each equation)
   \rho_s&:=& \mathrm{ind}_{\mathrm P_{r,n-r}}^{\mathrm G_n}\sigma\wo\tau_s \\
   &:=& \{f\in \con{\,^\infty}(\mathrm G_n; \sigma\wo\tau_s)\mid f(pg)=p.f(g),\,p\in \mathrm P_{r,n-r},\,g\in \mathrm G_n\},
\end{eqnarray*}
which is a Casselman-Wallach representation of $\mathrm G_n$ under
right translations. Here $\sigma$ is an  irreducible
Casselman-Wallach representation of $\mathrm G_r$, as in Theorem
\ref{main}, and $\sigma\wo\tau_s$ is viewed as a representation of
$\mathrm P_{r,n-r}$ by the inflation through the obvious homomorphism
$\mathrm P_{r,n-r}\rightarrow \mathrm G_r\times \mathrm G_{n-r}$.

Combining Langlands classification with the result of Speh-Vogan
(\cite{SV}), we have
\begin{prp}\label{prop:3.1}
The representation $\rho_s$ of $\mathrm G_n$ is irreducible except
for a measure zero set of $s\in \mathbb C$.
\end{prp}

As in the Theorem \ref{main}, let  $\chi$ be a generic character of
$\mathrm R_r$. In order to prove Theorem \ref{main}, replacing
$\sigma$ by its tensor product with a suitable character, we may
(and do) assume that
\begin{equation}\label{dec}
\textrm{the decent of $\chi$ to $\mathrm G_r\times \RN_{n-r}$ has
the form $1\otimes \psi$,}
 \end{equation}
where $1$ stands for the trivial character on $\mathrm G_r$, and
$\psi$ is a generic unitary character on $\oN_{n-r}$. Fix a non-zero
element
\[
  \lambda\in \mathrm{Hom}_{\oN_{n-r}}(\tau,{\psi^{-1}}).
\]
Define a continuous linear map
\[
 \Lambda: \sigma\wo\tau_s=\sigma\wo\tau\rightarrow \sigma,\qquad
 u\otimes v \mapsto \lambda(v) u  \qquad (s\in \mathbb{C}).
\]

Let $\pi$ be an irreducible Casselman-Wallach representation of
$\mathrm G_n$ as  in Theorem \ref{main},  and let
$$ \la \,, \,\ra_\mu:\pi\times \sigma\rightarrow \mathbb C$$
be a continuous bilinear map which represents an element $\mu\in
\mathrm{Hom}_{\mathrm R_r}(\pi\wo\sigma,\chi)$.

Write $\mathrm e_1, \mathrm e_2,\cdots,\mathrm e_n$ for the standard
basis of $\rk^n$. Note that the vector $\mathrm e_{r+1}$ is fixed by
the group $\RR_r$. Recall the representation $\omega$ of $\mathrm
G_n$ from the Introduction. Using the discussion of \cite[page
533]{Tr}, we identify $\rho_s\wo \omega$ with the space of
$\rho_s$-valued Schwartz functions on $\rk^n$. Then  \eqref{dec}
obviously implies the following
\begin{lemp}\label{inv} For every  $s\in \mathbb C, v\in \pi$ and $\phi\in \rho_s\wo
\omega$, the function
$$g\mapsto \la g.v,\Lambda\left(\phi(g^{-1}\mathrm e_{r+1})(g)\right)\ra_\mu $$
on $\RG_n$ is left $\mathrm R_r$-invariant.
\end{lemp}

The group $\mathrm R_r$ is not unimodular in general. Nevertheless,
we still have the following
\begin{lemp}\label{haar} (cf. \cite[Theorem 33D]{Lo}) Up to scalar multiplication, there exists a unique positive   Borel measure $\rd g$ on $\mathrm R_r\backslash \mathrm G_n$ such that
\[
  \int_{\mathrm R_r\backslash \mathrm G_n}\varphi(gg_0)\,\rd g=\abs{\det(g_0)}_{\rk}^{r+1-n} \int_{\mathrm R_r\backslash \mathrm G_n}\varphi(g)\,\rd g,
\]
for all $g_0\in \mathrm G_n$ and all non-negative measurable
function $\varphi$ on $\mathrm R_r\backslash \mathrm G_n$.
\end{lemp}

Lemma \ref{inv} allows us to define the following zeta integrals
which play a key role in the proof of Theorem \ref{main}:
\begin{equation}\label{zeta}
\operatorname Z_\mu(v,\phi):=\int_{\mathrm R_r\backslash \mathrm
G_n}\la g.v,\Lambda\left(\phi(g^{-1}\mathrm
e_{r+1})(g)\right)\ra_\mu\, \rd g,
\end{equation}
where $v\in \pi, \phi\in \rho_s\wo \omega$, and $\rd g$ is as in
Lemma \ref{haar}. It follows from Lemma \ref{haar} that this
integral has the following invariance property: \bee
\operatorname Z_\mu(g.v, g.\phi)=\abs{\det(g)}_{\rk}^{r+1-n}
\operatorname Z_\mu(v,\phi),\qquad g\in \RG_n. \eee \vspace{1mm}

The remainder of this section is devoted to the proof of Theorem A.
To this end, let us first state two results about the integral
$\oZ_\mu(v,\phi)$, which will be proved in Sections \ref{sec3} and
\ref{sec4} respectively.
\begin{prp}\label{prop1} If $\mu\ne 0$, then for every $s\in \mathbb C$, there exist $v\in \pi$ and $\phi\in \rho_s\wo \omega$
 such that the integral $\operatorname Z_\mu(v,\phi)$ is absolutely convergent and non-zero.
\end{prp}
\begin{prp}\label{prop2} There is a real constant $c_\mu$, depending on $\pi$, $\sigma$, $\chi$, $\tau$ and $\mu$, such that
for every $s\in \mathbb C$ whose real part $>c_\mu$, the integrals
in \eqref{zeta} are absolutely convergent, and produce a continuous
linear functional in  $\mathrm{Hom}_{\mathrm G_n}(\pi\wo \rho_s\wo
\omega,\abs{\det}_\rk^{r+1-n})$.
\end{prp}

We are now in a position to complete the proof of Theorem
\ref{main}. Let  $F$ be a finite dimensional subspace of
$\mathrm{Hom}_{\mathrm R_r}(\pi\widehat{\otimes}\sigma, \chi)$. By
Proposition \ref{prop1} and Proposition \ref{prop2}, we have an
injective linear map
\begin{equation}\label{zetain}
F\rightarrow \mathrm{Hom}_{\mathrm
G_n}(\pi\widehat{\otimes}\rho_s\widehat{\otimes}\omega,
\abs{\det}_\rk^{r+1-n}),\quad \mu\mapsto \oZ_\mu,
\end{equation}
for all $s\in \mathbb{C}$ with sufficiently large real part.  By
Proposition \ref{prop:3.1} and Theorem \ref{main2}, the hom space in
\eqref{zetain} is at most one dimensional  except
for a measure zero set of $s\in \mathbb C$. Therefore $F$ is  at most
one dimensional. This finishes the proof of Theorem \ref{main}.

\section{Preliminaries on Schwartz inductions}

For the proof of Proposition \ref{prop1}, we recall in this section some properties of Schwartz inductions (\cite[Section 2]{du}). We work in the setting of Nash manifolds and Nash groups. The reader is referred to \cite{Shi} or \cite{Sun2} for the notions of Nash manifolds, Nash maps, Nash submanifolds, and the related notion of semialgebraic sets.

By a \emph{Nash group}, we mean a group which is
simultaneously a Nash manifold so that all group operations (the
multiplication and the inversion) are Nash maps. Every semialgebraic subgroup
of a Nash group is automatically closed and is called a \emph{Nash subgroup}. Every Nash subgroup is canonically a Nash group.

An action of a Nash group $G$ on a Nash manifold $M$ is called a \emph{Nash action} if the action map $G\times M\rightarrow M$ is a Nash map. Likewise, a finite dimensional real representation $V_\R$ of a Nash group $G$ is called a \emph{Nash representation} if the action map $G\times
V_\R\rightarrow V_\R$ is a Nash map. A Nash group is said to be \emph{almost
linear} if it admits a Nash representation with finite kernel. Structures and basic properties of almost linear Nash groups are studied in detail in \cite{Sun2}. Recall that a Nash manifold is said to be \emph{affine} if it is Nash
diffeomorphic to a Nash submanifold of some $\R^n$. All almost linear Nash groups are affine as Nash manifolds.

Let $M$ be an affine Nash manifold. For each complex Fr\'echet space $V_0$, a $V_0$-valued smooth function $f\in \con^{\,\infty}(M; V_0)$ is said to be
\emph{Schwartz} if
\[
  |f|_{D,\nu}:=\sup_{x\in M} |(D f)(x)|_\nu<\infty
\]
for all Nash differential operators $D$ on $M$, and all continuous
seminorms $|\cdot|_\nu$ on $V_0$. Recall that a differential operator $D$
on $M$ is said to be Nash if $D\varphi$ is a Nash function whenever
$\varphi$ is a ($\C$-valued) Nash function on $M$. Denote by
$\con^{\,\varsigma}(M; V_0)\subset \con^{\,\infty}(M; V_0)$ the subspace of all Schwartz
functions. Then both $\con^{\,\varsigma}(M; V_0)$ and $\con^{\,\infty}(M; V_0)$ are
naturally Fr\'{e}chet spaces, and the inclusion map $\con^{\,\varsigma}(M;
V_0)\hookrightarrow \con^{\,\infty}(M; V_0)$ is continuous. Furthermore, we have
that (cf. \cite[page 533]{Tr})
\[\con^{\,\varsigma}(M; V_0)=\con^{\,\varsigma}(M)\widehat \otimes  V_0\qquad \text{and}\qquad
\con^{\,\infty}(M; V_0)=\con^{\,\infty}(M)\widehat \otimes  V_0,\]
where $\con^{\,\varsigma}(M):=\con^{\,\varsigma}(M;\C)$ and  $\con^{\,\infty}(M):=\con^{\,\infty}(M;\C)$.

Let $G$ be an almost linear Nash group, $S$ a Nash subgroup of $G$, and let $V_0$ be a smooth Fr\'{e}chet
representation of $S$ of moderate growth  (for the usual notion of  smooth Fr\'{e}chet
representation of moderate growth, see \cite[Definition 1.4.1]{du} or \cite[Section 2]{S}, for example).  We define the un-normalized Schwartz induction $\ind_S^G\, V_0$ to be the image of the following continuous linear map:
\begin{equation*}
  \begin{array}{rcl}
  i_{S, V_0}: \con^{\,\varsigma}(G; V_0)&\rightarrow &\con^{\,\infty}(G; V_0),\\
   f&\mapsto &\left(\, g\mapsto \int_S s.f(s^{-1}g)\,\rd s\,\right),
  \end{array}
\end{equation*}
where $\rd s$ is a left invariant Haar measure on $S$.
Under the quotient topology of $\con^{\,\varsigma}(G; V_0)$ and under right translations, this is a smooth Fr\'{e}chet   representation of $G$ of moderate growth.

We will use the following basic properties of Schwartz inductions for almost linear Nash groups.
\begin{lem}\label{Schwartz} Let $G$ be an almost linear Nash group, $S$ a Nash subgroup of $G$, and let $V_0$ be a smooth Fr\'{e}chet
representation of $S$ of moderate growth.

(a). For each Nash subgroup $S'$ of $G$ containing $S$, the map
\[
  \begin{array}{rcl}
\ind^G_S\, V_0&\rightarrow &\ind^G_{S'}\ind^{S'}_S V_0,\smallskip \\
f&\mapsto&  \left(g\mapsto (s'\mapsto f(s'g))\right)
\end{array}
\]
is well-defined and is an isomorphism of representations of $G$.

(b). For each Nash subgroup $T$ of $G$ such that $ST$ is open in $G$, the map
\[
  \begin{array}{rcl}
\ind_{S\,\cap\, T}^T (V_0|_{S\,\cap \,T})&\rightarrow &(\ind_{S}^G \,V_0)|_T, \smallskip \\
   \phi&\mapsto &\left( g\mapsto \begin{cases} s.\phi(t),\ &\text{if}\ g=st\in S\,T;\\
     0,\ &\text{otherwise}
                               \end{cases}
                               \right)
  \end{array}
\]
is well-defined and is an injective homomorphism of representations
of $T$.

(c). Let $V$ be a smooth Fr\'echet representation of $G$ of moderate growth, and let $G\times M\rightarrow M$ be a transitive Nash action of $G$ on an affine Nash manifold $M$.
Assume that $V$ is nuclear as a Fr\'echet space, and $S$ equals the stabilizer in $G$ of a point $x_0\in M$. Then the map
\[
  \begin{array}{rcl}
    \ind_S^G (V|_S) &\rightarrow & \con^{\,\varsigma}(M;V),\\
       \phi&\mapsto &(g.x_0\mapsto g.\phi(g^{-1}))
  \end{array}
\]
is well-defined and is an isomorphism of representations of $G$.
Here the action of $G$ on $\con^{\,\varsigma}(M; V)$ is given by
\[ (g.\psi)(x):=g.(\psi(g^{-1}.x)),\qquad \psi\in \con^{\,\varsigma}(M;V), \,g\in G,\, x\in M.\]

(d). Let $f\in \con^{\,\infty}(G;V_0)$. If
 \[
   f(sg)=s.f(g),\quad s\in S, \,g\in G,
 \]
and $f$ is compactly supported modulo $S$, then $f\in \ind_S^G V_0$.
\end{lem}

\begin{proof} Part (a) of the lemma is \cite[Lemma 2.1.6]{du}, and part (b)
follows by the arguments of \cite[Section 2.2]{du}. Part (c) is a special case of \cite[Lemma 3.2]{LS}, and part (d) is \cite[Lemma 3.1]{LS}.
\end{proof}

\section{Proof of Proposition \ref{prop1}}\label{sec3}

We continue with the notation of Section \ref{sec2}. We first
present a measure decomposition for the measure $\rd g$ given in
Lemma \ref{haar}. For every $m\geq 1$, write
\be\label{mirop}
  \mathrm P_m:=\left\{\begin{bmatrix}1_1 & * \\ 0 & *\end{bmatrix}\in \mathrm G_m \right\} \quad \textrm{and}\quad  \mathrm Q_m:=\left\{\begin{bmatrix}* & 0 \\ * & 1_{m-1}\end{bmatrix}\in \mathrm G_m \right\}.
\ee
Respectively write $\mathrm G_{n-r}'$, $\mathrm P_{n-r}'$,
$\mathrm Q_{n-r}'$ and $\mathrm N_{n-r}'$ for the image of $\mathrm
G_{n-r}$, ${\mathrm P_{n-r}}$, $\mathrm Q_{n-r}$ and $\mathrm
N_{n-r}$, under the embedding
\[
  \mathrm G_{n-r}\rightarrow \mathrm G_n,\quad g\mapsto \begin{bmatrix}1_r & 0 \\ 0 & g\end{bmatrix}.
\]
Put
\begin{equation}\label{urdef}
  \mathrm U_{r}:=\left\{\begin{bmatrix}1_r & u \\ 0 & 1_{n-r}\end{bmatrix}|\,
u=\begin{bmatrix}\ast & 0 &\cdots &0  \\ \vdots & \vdots & &\vdots\\
\ast & 0 &\cdots &0\end{bmatrix}\right\}
\end{equation}
and
\[
 \bar{\mathrm U}_{r,n-r}:=\left\{\begin{bmatrix}1_r & 0 \\ * &
1_{n-r}\end{bmatrix}\right\}.
\]

\begin{lem}\label{decm}
 The multiplication map
\begin{equation}\label{mul}
   \mathrm  U_r\times (\oN'_{n-r}\backslash \mathrm P_{n-r}')\times \mathrm Q_{n-r}'\times \bar{\mathrm U}_{r,n-r}\rightarrow \mathrm R_r\backslash \mathrm G_n
\end{equation}
is a well-defined open embedding, and its image has full measure in
$\mathrm R_r\backslash \mathrm G_n$. Moreover, the restriction of
the measure $\rd g$ (see Lemma \ref{haar}) through the embedding
\eqref{mul} has the form $\rd u\otimes \rd p' \otimes \rd q'\otimes
\rd \bar u$,  where   $\rd u$, $\rd p'$, $\rd q'$ and $\rd \bar u$
are positive Borel measures on  $\mathrm  U_r$,
$\oN'_{n-r}\backslash \mathrm P_{n-r}'$, $\mathrm Q'_{n-r}$ and
$\bar{\mathrm U}_{r,n-r}$, respectively.
\end{lem}
\begin{proof}
Note that the multiplication map
\begin{equation}\label{mul2}
   (\mathrm R_r\backslash \mathrm P_{r, n-r})\times \bar{\mathrm U}_{r,n-r} \rightarrow \mathrm R_r\backslash \mathrm G_n
\end{equation}
is an open embedding whose image has full measure in $\mathrm
R_r\backslash \mathrm G_n$. The relative-invariance of $\rd g$
implies that its restriction through \eqref{mul2} has the form $\rd
p\otimes \rd \bar u$, where $\rd \bar u$ is a Haar measure on
$\bar{\mathrm U}_{r,n-r}$, and $\rd p$ is a positive Borel measure
on $\mathrm R_r\backslash \mathrm P_{r, n-r}$ which is
relative-invariant under right translations of $\mathrm G_{n-r}'$.

Note that
\[
  u\oN_{n-r}'u^{-1}\subset \mathrm R_r\quad \textrm{for all }u\in \mathrm U_r.
\]
Therefore the multiplication map
\begin{equation}\label{mul3}
  \mathrm U_r\times   (\oN_{n-r}'\backslash \mathrm G_{n-r}') \rightarrow \mathrm R_r\backslash \mathrm P_{r, n-r}
\end{equation}
is well-defined. Easy calculation shows that it is in fact a
diffeomorphism. By $\mathrm G_{n-r}'$-relative-invariance, the pull
back of $\rd p$ through \eqref{mul3} has the form $\rd u\otimes \rd
g'$, where $\rd u$ is a positive Borel measure on $\mathrm U_r$, and
$\rd g'$ is a relative-invariant positive Borel measure on
$\oN_{n-r}'\backslash \mathrm G_{n-r}'$.

Now, the multiplication map
\begin{equation}\label{mul4}
 (\oN'_{n-r}\backslash \mathrm
P_{n-r}')\times \mathrm Q_{n-r}'\rightarrow \oN_{n-r}'\backslash
\mathrm G_{n-r}'
\end{equation}
is an open embedding whose image has full measure in
$\oN_{n-r}'\backslash \mathrm G_{n-r}'$. The relative-invariance of
$\rd g'$ implies that its restriction through \eqref{mul4} has the
form $\rd p'\otimes \rd q'$, where $\rd p'$  and $\rd q'$ are as in
the Lemma. This finishes the proof.
\end{proof}

Denote by $\RP^\circ_{r+1}$  the stabilizer of $\mathrm e_{r+1}\in
\rk^n$ in $\mathrm G_n$. It consists all matrices in $\mathrm G_n$
whose $(r+1)$-th column equals $\mathrm e_{r+1}$.

\begin{lem}\label{key} There exists an injective homomorphism
\[\eta: \ind_{\RP_{r,n-r}\cap \RP^\circ_{r+1}}^{\RG_n}\sigma \wo\tau_s\rightarrow \con^{\,\varsigma}(\rk^n;\rho_s)\]
of representations of $\RG_n$ such that
\be(\eta(\phi)(g^{-1}\mathrm e_{r+1}))(g)=\phi(g)\ee
for every $\phi\in \ind_{\RP_{r,n-r}\cap \RP^\circ_{r+1}}^{\RG_n} \sigma \wo\tau_s$ and $g\in \RG_n$.
\end{lem}
\begin{proof} Applying part (a) of Lemma \ref{Schwartz}, we get the following isomorphism of representations of $\RG_n$:
$$
  \begin{array}{rcl}
\eta_1: \ind_{\RP_{r,n-r}\cap \RP^\circ_{r+1}}^{\RG_n}\sigma
\wo\tau_s &\rightarrow &
\ind^{\RG_n}_{\RP^\circ_{r+1}}\ind^{\RP^\circ_{r+1}}_{\RP_{r,n-r}\cap
\RP^\circ_{r+1}}\sigma \wo\tau_s,\\
 \phi& \mapsto & (g\mapsto (p^\circ
\mapsto \phi(p^\circ g))).
\end{array}
$$
 Since $\RP_{r,n-r}\RP^\circ_{r+1}$ is open in $\RG_n$, it follows from part (b) of Lemma \ref{Schwartz} that there is an injective homomorphism
\begin{eqnarray*}
\eta_2:\ind^{\RG_n}_{\RP^\circ_{r+1}}\ind^{\RP^\circ_{r+1}}_{\RP_{r,n-r}\cap \RP^\circ_{r+1}}\sigma \wo\tau_s\rightarrow
\ind^{\RG_n}_{\RP^\circ_{r+1}}(\ind^{\RG_n}_{\RP_{r,n-r}}\sigma\wo \tau_s)|_{\RP^\circ_{r+1}},
\end{eqnarray*}
such that
\bee ((\eta_2(\phi))(g))(g_0)=\begin{cases} p.(\phi(g)(p^\circ)),\ &\text{if}\ g_0=pp^\circ \in \RP_{r,n-r} \RP^\circ_{r+1};\\
0,\ &\text{otherwise}.
\end{cases}
\eee

Notice that $\RG_{n}$ acts transitively on $\rk^n\setminus\{0\}$.
Using part (c) of Lemma \ref{Schwartz}, we obtain the following
isomorphism:
\[
  \begin{array}{rcl}
\eta_3: \ind^{\RG_n}_{\RP^\circ_{r+1}}\rho_s|_{\RP^\circ_{r+1}}
 &\rightarrow & \con^{\,\varsigma}(\rk^n\setminus\{0\} ;\rho_s),\\
\psi & \mapsto & (g\mathrm{e}_{r+1}\mapsto g.\psi(g^{-1})).
 \end{array}
\] Denote
by $\eta$ the injective homomorphism
\[\eta_4 \circ \eta_3 \circ
\eta_2\circ \eta_1: \ind_{\RP_{r,n-r}\cap
\RP^\circ_{r+1}}^{\RG_n}\sigma \wo\tau_s\rightarrow
\con^{\,\varsigma}(\rk^n;\rho_s),\] where
$\eta_4:\con^{\,\varsigma}(\rk^n\setminus\{0\} ;\rho_s)
\hookrightarrow \con^{\,\varsigma}(\rk^n;\rho_s)$ is the map
obtained by extension by zero.  Then, for every $\phi\in
\ind_{\RP_{r,n-r}\cap \RP^\circ_{r+1}}^{\RG_n}\sigma \wo\tau_s$ and
$g\in \RG_n$, we have that
\begin{equation*}\begin{split}
&(\eta(\phi)(g^{-1}\mathrm{e}_{r+1}))(g)=(g^{-1}.(((\eta_2\circ
\eta_1)(\phi))(g)))(g)\\=& (((\eta_2\circ
\eta_1)(\phi))(g))(1)=((\eta_1(\phi))(g))(1)=\phi(g),
\end{split}\end{equation*}
as desired.
 \end{proof}

As in Proposition \ref{prop1}, let $\mu$ be a non-zero element in
$\mathrm{Hom}_{\mathrm R_r}(\pi\wo\sigma,\chi)$. Fix $v_\pi\in \pi$
and $v_\sigma\in \sigma$ such that
\[
  \la v_\pi, v_\sigma\ra_\mu\neq 0.
\]
Note that the multiplication map
$$
  (\mathrm P_{r,n-r}\cap \RP^\circ_{r+1})\times \mathrm  U_r\times \mathrm Q_{n-r}'\times \bar{\mathrm U}_{r,n-r}\rightarrow \RG_n
$$
is an open embedding of smooth manifolds. Let $\phi_0$ be a
compactly supported smooth function on $\mathrm U_r\times \mathrm
Q_{n-r}'\times \bar{\mathrm U}_{r,n-r}$, and let $v_\tau$ be a
vector in $\tau$. Define a $\sigma \wo\tau_s$-valued smooth function
$\phi$ on $\RG_n$ by
\[
  \phi(g)=\begin{cases}\phi_0(u,q',\bar u) \cdot (p.(v_\sigma\otimes
  v_\tau)), \ &\text{if}\ g=p uq'\bar u\\
    &\in  (\mathrm P_{r,n-r}\cap \RP^\circ_{r+1}) \mathrm  U_r \mathrm Q_{n-r}' \bar{\mathrm U}_{r,n-r};\\
0,\ &\text{otherwise}.
\end{cases}
\]
Part (d) of Lemma \ref{Schwartz} implies that $\phi\in
\mathrm{ind}_{\mathrm P_{r,n-r}\cap \RP^\circ_{r+1}}^{\mathrm
G_n}\sigma \wo\tau_s$.

It is easy to check that
\[
  p'^{-1}up'u^{-1}\in  \mathrm U_{r,n-r}\cap \RP^\circ_{r+1} \qquad
  \textrm{for all } p'\in \mathrm
P_{n-r}', \,u\in \mathrm  U_r,
\]
where $\mathrm U_{r,n-r}$ denotes the unipotent radical of $\mathrm
P_{r,n-r}$. Therefore,
\be\label{deco}
  \phi(u p' q' \bar u)=\phi(p'(p'^{-1}up'u^{-1}) u q'\bar u)=\phi_0(u,q',\bar u)\cdot  (p'.(v_\sigma\otimes v_\tau)),
\ee for all  $\,u\in \mathrm  U_r$, $p'\in \mathrm P_{n-r}'$,
$q'\in\mathrm Q_{n-r}'$ and  $\bar u\in \bar{\mathrm U}_{r,n-r}$.

For simplicity in notation, put
\[
  \Omega:= \mathrm U_r\times (\oN'_{n-r}\backslash \mathrm
P_{n-r}')\times
    Q_{n-r}'\times \bar{\mathrm U}_{r,n-r}.
\]
 Recall the homomorphism
$\eta$ of Lemma \ref{key}. Then we have that
\begin{eqnarray}
 \label{f16}& &\oZ_\mu(v_\pi,\eta(\phi))\\
\nonumber &=& \int_\Omega \la u p' q' \bar u.v_\pi,\Lambda\big(\eta(\phi)((u p' q' \bar u)^{-1}
\mathrm{e}_{r+1})(u p' q' \bar u)\big)\ra_\mu\, \rd u \,\rd p' \,\rd q'\,\rd \bar u \\
 \nonumber & &\qquad \qquad \qquad \qquad \qquad \qquad \, (\text{by Lemma \ref{decm}})\\
\nonumber &=& \int_\Omega \la u p' q' \bar u.v_\pi,
    (\Lambda\circ \phi)(u p' q' \bar u)\ra_\mu\, \rd u \,\rd p' \,\rd q'\,\rd \bar u \\
\nonumber & &\qquad \qquad \qquad  \qquad \qquad \qquad (\text{by
Lemma
 \ref{key}})\\
\nonumber     &=& \int_\Omega \la u p' q' \bar u.v_\pi,
v_\sigma\ra_\mu \,
    \lambda(p'. v_\tau)\,\phi_0(u,q',\bar u)\, \rd u \,\rd p' \,\rd q'\,\rd \bar u.\\
 \nonumber   & &\qquad \qquad \qquad \qquad \qquad \qquad (\text{by \eqref{deco}})
  \end{eqnarray}
Here in the last term of \eqref{f16}, the representation $\tau$ of
$\mathrm G_{n-r}$ is viewed as a representation of $\mathrm
G'_{n-r}$ via the obvious isomorphism $\mathrm G_{n-r}\cong \mathrm
G'_{n-r}$. Likewise, we view the character $\psi$ of $\mathrm
N_{n-r}$ as a character of $\mathrm N_{n-r}'$. Recall from
\cite[Section 3]{JS1} that for each smooth function $W$ on $\mathrm
P_{n-r}'$, if
\[
  W(np')=\psi(n)^{-1}W(p'), \quad n\in \mathrm N_{n-r}',\, p'\in \mathrm
  P_{n-r}',
  \]
and $W$ has compact support modulo $\mathrm N_{n-r}'$, then there is
a vector $v\in \tau$ such that
\[
  W(p')=\lambda(p'.v),\quad p'\in  \mathrm P_{n-r}'.
\]
Thus, by \eqref{f16}, the integral $\oZ_\mu(v_\pi,\eta(\phi))$ is
absolutely convergent and non-zero for appropriate $v_\tau$ and
$\phi_0$. This completes the proof of Proposition \ref{prop1}.

\section{Proof of Proposition \ref{prop2}}\label{sec4}

For every $\ft=(t_1,\cdots,t_{n-r})\in (\BR_+^\times)^{n-r}$ and
$\fs=(s_1,\cdots,s_r)\in \rk^r$, put
\[
a_\ft:=\diag\{1,\cdots,1,t_1,\cdots,t_{n-r}\}\in \mathrm G_{n-r}'
\]
and
\[
 b_\fs:=\begin{bmatrix}1_r & u_\fs \\ 0 & 1_{n-r}\end{bmatrix}\in \mathrm U_r, \quad \textrm{where }\,
u_\fs=\begin{bmatrix}s_1 & 0 &\cdots &0  \\ \vdots & \vdots & &\vdots\\
s_r & 0 &\cdots &0\end{bmatrix}.
\]

Fix an arbitrary maximal compact subgroup $\RK_n$ of $\RG_n$. It is
easy to check that there is a positive character $\gamma$ on
$(\BR_+^\times)^{n-r}$ such that \be\label{eq:4.1}
\int_{\RR_r\backslash \RG_n}\varphi(g)\,\rd g=\int_{\rk^{r}\times
(\RR_+^\times)^{n-r}\times \RK_n} \gamma(\ft)\,\varphi(a_\ft b_\fs
k)\,\rd\fs\,\rd^\times \ft\,\rd k, \ee for all non-negative
continuous functions $\varphi$ on $\RR_r\backslash \RG_n$. Here $\rd
k$ is the normalized Haar measure on $\RK_n$; and $\rd \fs$ and
$\rd^\times \ft$ are appropriate Haar measures on $\rk^{r}$ and
$(\BR_+^\times)^{n-r}$, respectively.

For each $ g\in \RG_n$, write
$$||g||:=\mathrm{Tr}(g\ ^t\bar{g})+\mathrm{Tr}((g\ ^t\bar{g})^{-1}), \qquad$$
where $\,^t\bar{g}$ denotes the conjugate transpose of $g$. For each
$\ft=(t_1,\cdots,t_{n-r})\in (\BR_+^\times)^{n-r}$, put
$$
  ||\ft||:=\prod_{i=1}^{n-r}(t_i+t_i^{-1}).
$$
Note that there exists a constant $c$, which depends on $n$ only,
such that
$$||a_\ft||\le c\,||\ft||^2,\qquad \ft\in (\BR^\times_+)^{n-r}.$$

Take a positive integer $c_0$ such that
\be\label{eq:4.2}
 \gamma(\ft)\le ||\ft||^{c_0},\qquad  \ft\in (\BR_+^\times)^{n-r}.
\ee Since the bilinear form $\la\, ,\,\ra_\mu$ is continuous, there
exist  continuous seminorms $|\cdot|_{\pi,1}$  and $|\cdot|_\sigma$
on  $\pi$ and $\sigma$ respectively such that \be\label{eq:4.3}
 |\la u,v\ra_\mu|\le |u|_{\pi,1}\cdot |v|_\sigma, \qquad
  u\in \pi,\, v\in \sigma.
\ee Using the moderate growth condition on $\pi$, we get two
positive integers $d$ and $c_1$, and a continuous seminorm
$|\cdot|_{\pi,2}$ on $\pi$ so that
\be
  \label{eq:4.4} |(a_\ft b_\fs
k).u|_{\pi,1}\le ||b_\fs||^d\cdot ||\ft||^{c_1}\cdot |u|_{\pi,2},
\ee
for all $\fs\in \rk^{r}$, $\ft\in (\BR_+^\times)^{n-r}$, $k\in
\RK_n$ and $u\in \pi$. Since $\RU_r$ acts trivially on $\sigma\wo
\tau_s$, one has that \be\label{eq:4.5.} f(a_\ft b_\fs  k)=a_\ft
b_\fs.f(k)=f(a_\ft k),\ee for all $\fs\in \rk^{r}$, $\ft\in
(\BR_+^\times)^{n-r}$, $k\in \RK_n$ and $f\in \rho_s.$

Now, for every $u\in \pi$, $f\in \rho_s$ and $\phi\in \omega$, we
have that
\begin{eqnarray} \label{eq:a1}&&|\oZ_\mu(u,f\otimes\phi)|\\
\nonumber&\le &\int_{\RR_r\backslash \RG_n}|\la g.u,(\Lambda\circ f)(g)\ra_\mu|\cdot|\phi(g^{-1}\mathrm e_{r+1})|\,\rd g \\
\nonumber&=&\int_{\rk^{r}\times (\BR_+^\times)^{n-r}\times \RK_n}
\gamma(\ft)\cdot
|\la ( a_\ft b_\fs k).u,(\Lambda\circ f)( a_\ft b_\fs k)\ra_\mu|\\
\nonumber&&\qquad \qquad \qquad\,  \cdot |\phi(( a_\ft b_\fs
k)^{-1}e_{r+1})|\,\rd \fs
\,\rd^\times \ft \,\rd k \qquad\quad (\text{by}\ \eqref{eq:4.1})\medskip \\
\nonumber&\le&\int_{\rk^{r}\times (\BR_+^\times)^{n-r}\times \RK_n}
||\ft||^{c_0}\cdot
|( a_\ft b_\fs k).u|_{\pi,1}\cdot|(\Lambda\circ f)(a_\ft b_\fs k)|_\sigma\\
\nonumber& &\qquad \qquad \qquad\, \cdot|\phi(( a_\ft b_\fs
k)^{-1}e_{r+1})|\,\rd \fs
\,\rd^\times \ft\, \rd k \quad\qquad (\text{by}\ \eqref{eq:4.2}\ \text{and}\ \eqref{eq:4.3})\medskip\\
\nonumber&\le&|u|_{\pi,2}\cdot \int_{\rk^{r}\times
(\BR_+^\times)^{n-r}\times \RK_n} ||\ft||^{c_0+c_1}\cdot
||b_\fs||^d\cdot
|(\Lambda\circ f)(a_\ft k)|_\sigma\\
\nonumber&& \qquad \qquad \qquad\, \cdot | \phi(( a_\ft b_\fs
k)^{-1}e_{r+1})|\,\rd \fs \,\rd^\times \ft \,\rd k. \quad\qquad
(\text{by}\ \eqref{eq:4.4}\ \text{and}\ \eqref{eq:4.5.})
\end{eqnarray}

\

Recall that the  linear map $\Lambda:\sigma\wo \tau\rightarrow
\sigma$ is continuous. Then it follows from the moderate growth
condition on $\sigma\wo\tau$ that there exist a positive integer
$c_2$ and a continuous seminorm $|\cdot|_{\sigma\wo \tau}$ on
$\sigma\wo \tau$ such that
\[
  |\Lambda(p.u)|_{\sigma}\le
||p||^{c_2}\cdot|u|_{\sigma\wo \tau},\qquad p\in \RP_{r,n-r},\,  u\in
\sigma\wo \tau.
\]

\begin{lem}\label{lem:5.1} For every  positive integer
$N$, there is a continuous seminorm $|\cdot|_{\sigma\wo \tau,N}$ on
$\sigma\wo \tau$ (which depends on $|\cdot|_{\sigma\wo \tau}$ and
$N$) such that \[ |\Lambda(a_\ft.u)|_\sigma\le
\xi(\ft)^{-N}\cdot||\ft||^{c_2}\cdot |u|_{\sigma\wo \tau,N}, \quad
u\in \sigma\wo \tau,\,\,\ft\in (\BR_+^\times)^{n-r}\] where
\[
\xi(\ft):=\prod_{i=1}^{n-r-1} (1+\frac{t_i}{t_{i+1}})\quad
\textrm{for } \ft=(t_1,\cdots,t_{n-r}).\]
\end{lem}
\begin{proof}
This is proved in \cite[Lemma 6.2]{JSZ}. \end{proof}

Lemma \ref{lem:5.1} easily implies that
 \be\label{eq:4.5}
|(\Lambda\circ f)(a_\ft k)|_\sigma\le
\Pi(\ft)^{-\mathrm{Re}s}\cdot\xi(\ft)^{-N}\cdot||\ft||^{c_2}\cdot
|f(k)|_{\sigma\wo \tau,N}, \ee for all $f\in \rho_s$, $\ft\in
(\BR_+^\times)^{n-r}$ and  $k\in \RK_n$. Here
\[
\Pi(\ft):=\prod_{i=1}^{n-r}t_i\quad \textrm{for }
\ft=(t_1,\cdots,t_{n-r}).\] Let $|\cdot|_{\sigma\wo \tau,N}$ be  as
in Lemma \ref{lem:5.1}. Define a continuous seminorm
$|\cdot|_{\rho_s,N}$ on $\rho_s$ by
\begin{equation}\begin{split}\label{eq:a4}
|f|_{\rho_s,N}:= \mathrm{sup}\,\{\,|f(k)|_{\sigma\wo\tau,N}\mid k\in
\RK_n\,\},\quad f\in \rho_s.
\end{split}\end{equation}

Set $c_\mu:=c_0+c_1+c_2$.

\begin{lem}\label{lem:5.2} If
$\mathrm{Re}s>c_\mu$, then for all large enough integer $N$, one has
that \be\label{eq:4.6}
c_{s,N}:=\int_{(\BR_+^\times)^{n-r}}||\ft||^{c_\mu}\cdot\Pi(\ft)^{-\mathrm{Re}s}\cdot
\xi(\ft)^{-N}\cdot(1+t_1^{-1})^{-N}\,\rd^\times \ft<\infty.\ee
\end{lem}

\begin{proof}
The proof is similar to that of \cite[Lemma 6.2]{LS}. We omit the
details.
\end{proof}

For each positive integer $N$, put
\begin{equation}\begin{split}\label{eq:4.9} |\phi|_{\omega,N}\,=
\,\mathrm{sup}\,\{& \,(1+|a_{r+1}|_\rk)^N\cdot(|a_1|_\rk^2+\cdots+|a_r|_\rk^2)^N \cdot|\phi(k\mathbf a)|\\
 &\ | \ \mathbf a=\sum_{i=1}^n a_i\mathrm e_i\in \rk^n, k\in \RK_n\,\},\qquad \phi\in \omega.
\end{split}\end{equation}
This defines a continuous seminorm on $\omega$.

For each $\fs=(s_1,\cdots,s_{r})\in \rk^{r}$,
$\ft=(t_1,\cdots,t_{n-r})\in (\BR_+^\times)^{n-r}$ and $k\in \RK_n$,
it is routine to check that
\[ (a_\ft b_\fs k)^{-1}\mathrm
e_{r+1}=k^{-1}(t_{1}^{-1}\mathrm e_{r+1}-\sum_{i=1}^{r}s_i\mathrm
e_{i}).
\]
Together with \eqref{eq:4.9}, this implies that
 \be\label{eq:4.10}
(|s_1|_\rk^2+\cdots+|s_r|_\rk^2)^N\cdot (1+t_{1}^{-1})^N\cdot |\phi(
(a_\ft b_\fs k)^{-1}\mathrm e_{r+1})|\le |\phi|_{\omega,N} \ee for
all $\phi\in \omega$.

 In what follows we assume that
$\mathrm{Re}s>c_\mu$. Fix a positive integer $N$ which is large
enough so that \eqref{eq:4.6} holds true,  and  that
\be\begin{split}\label{eq:4.8} c_N':=\int_{\rk^{r}} ||b_\fs||^d\cdot
(|s_1|_\rk^2+\cdots+|s_r|_\rk^2)^{-N}\rd \fs <\infty.\qquad
(\fs=(s_1,\cdots,s_{r}))
\end{split}\ee
Recall the continuous seminorms $|\cdot|_{\sigma\wo \tau,N}$ on
$\sigma\wo \tau$ (see Lemma  \ref{lem:5.1}), $|\cdot|_{\rho_s,N}$ on
$\rho_s$ (see \eqref{eq:a4}) and $|\cdot|_{\omega,N}$ on $\omega$
(see \eqref{eq:4.9}). Then one has that
\begin{eqnarray}\label{eq:a2}
& &||\ft||^{c_0+c_1}\cdot ||b_\fs||^d\cdot|(\Lambda\circ f)(a_\ft
k)|_\sigma\cdot | \phi(( a_\ft b_\fs
k)^{-1}e_{r+1})| \medskip\\
\nonumber &\le& |f|_{\rho_s,N}\cdot ||\ft||^{c_\mu}\cdot \Pi(\ft)^{-\mathrm{Re}s}\cdot \xi(\ft)^{-N}\cdot ||b_\fs||^d
 \cdot | \phi(( a_\ft b_\fs
k)^{-1}e_{r+1})|\quad  (\text{by}\ \eqref{eq:4.5},\eqref{eq:a4}) \medskip \\
\nonumber &\le& |\phi|_{\omega,N} \cdot |f|_{\rho_s,N}\cdot
||\ft||^{c_\mu}\cdot\Pi(\ft)^{-\mathrm{Re}s}
\cdot\xi(\ft)^{-N}\cdot(1+t_1^{-1})^{-N}\medskip \\
\nonumber &&\cdot ||b_\fs||^d\cdot
(|s_1|_\rk^2+\cdots+|s_r|_\rk^2)^{-N}, \qquad\qquad \qquad \qquad
\qquad \qquad \qquad \, (\text{by}\ \eqref{eq:4.10})
\end{eqnarray}
for all $\fs=(s_1,\cdots,s_{r})\in \rk^{r}$,
$\ft=(t_1,\cdots,t_{n-r})\in (\BR_+^\times)^{n-r}$, $k\in \RK_n$,
$f\in \rho_s$ and $\phi\in\omega$.

Finally, one concludes from  \eqref{eq:a1}, \eqref{eq:a2},
\eqref{eq:4.6} and \eqref{eq:4.8} that
\begin{eqnarray*}|\oZ_\mu(u,f\otimes\phi)|
\le c_{N}'\cdot
c_{s,N}\cdot|u|_{\pi,2}\cdot|f|_{\rho_s,N}\cdot|\phi|_{\omega,N}
\end{eqnarray*}
 for all $u\in \pi$, $f\in
\rho_s$ and $\phi\in \omega$. This proves Proposition \ref{prop2}.

\section{The non-archimedean case}\label{sec5}
Let $\rk$ be a non-archimedean local field of characteristic  zero
throughout this section. Fix $n>r\geq 0$ as in the Introduction. We
shall prove an analog of Theorem \ref{main} in the non-archimedean
case:
\begin{introtheorem}\label{main3}
For all irreducible admissible  smooth  representations $\pi$ of
$\GL_n(\rk)$, and $\sigma$ of $\GL_r(\rk)$, and for all generic
characters $\chi$ of $\mathrm R_r$, one has that
\begin{equation*}
  \mathrm{dim}\ \mathrm{Hom}_{\mathrm R_r}(\pi \otimes \sigma,\chi)\le 1.
\end{equation*}
\end{introtheorem}
Here the Rankin-Selberg subgroup $\mathrm R_r$ of $\GL_n(\rk)$ is
defined as in \eqref{gp} and the generic characters of $\mathrm R_r$
are defined in the same way as in the archimedean case. As in
Introduction, let $\omega$ denote the representation of $\GL_n(\rk)$
on the space of $\C$-valued Schwartz-Bruhat functions on $\rk^n$.
Similar to what we have done in the archimedean case, we shall prove
Theorem \ref{main3} by reducing it to the following multiplicity one
theorem:

\begin{introtheorem}\label{main4} (cf. \cite[Theorem B]{S})
For all irreducible  admissible smooth  representations $\pi$ and
$\pi'$ of $\GL_n(\rk)$, and all characters $\chi'$ on $\GL_n(\rk)$,
one has that
\begin{equation*}
  \mathrm{dim}\ \mathrm{Hom}_{\GL_n(\rk)}(\pi \otimes \pi' \otimes  \omega,\chi')\le 1.
\end{equation*}
\end{introtheorem}

Now we turn to the proof of Theorem  \ref{main3}. The method of our
proof for Theorem \ref{main} remains valid. However, we will exploit
a different (and simpler) approach, which is developed in
\cite{GGP}.  As in Section \ref{sec2}, we write $\mathrm
G_m:=\GL_m(\rk)$ and $\RN_m:=\RN_m(\rk)$ for every non-negative
integer $m$.  For each smooth representation $\varrho$ of a closed
subgroup of $\RG_n$, write $\varrho^\vee$ for its contragradient
representation. In this section, we use $``\mathrm{ind}"$ to
indicate the un-normalized smooth induction with compact support.

Let $\pi$, $\sigma$ and $\chi$ be as in Theorem \ref{main3}.
Following \cite[Section 15]{GGP}, we fix a supercuspidal
representation $\tau$ of $\RG_{n-r}$ such that for all Levi
subgroups $M$ of $\RG_r$ and all  supercuspidal representations
$\mu$ of $M$, $\pi^\vee$ does not belong to the Bernstein component
(see \cite[Section 2.2]{B}, for example) associated to $(M\times
\RG_{n-r},\mu\otimes \tau)$. Without loss of generality, we assume
that $\rho:=\mathrm{ind}_{\RP_{r,n-r}}^{\RG_n}\sigma\otimes \tau$ is
irreducible (otherwise, using \cite[Theorem 3.2]{Sa}, replace $\tau$
by its twist by a suitable  unramified character). Here
$\RP_{r,n-r}$ is the parabolic subgroup of $\RG_n$ as in \eqref{pb}, and $\sigma\otimes\tau$ is viewed as a representation of
$\mathrm P_{r,n-r}$ by inflation through the homomorphism $\mathrm
P_{r,n-r}\rightarrow \RG_r\times \RG_{n-r}$.

As in \eqref{dec}, we assume that the decent of $\chi$ on
$\RG_r\times \RN_{n-r}$ has the form $1\otimes \psi$. We claim that
\begin{equation*}
\mathrm{Hom}_{\mathrm R_r}(\pi \otimes \sigma,\chi)\cong
\mathrm{Hom}_{\mathrm G_n}(\pi \otimes \rho\otimes
\omega,|\det|_\rk^{r+1-n}).
\end{equation*}
Using Theorem \ref{main4}, this clearly implies Theorem \ref{main3}. \\

\emph{Proof of the claim:} Write $\omega^\circ $ for the space of
$\mathbb C$-valued Schwartz-Bruhat functions on $\rk^n\setminus
\{0\}$, which is a smooth representation of $\RG_n$ as in
\eqref{action}. Then we have an obvious exact sequence
\[0\rightarrow \omega^\circ\otimes \rho\rightarrow \omega\otimes \rho\rightarrow \rho\rightarrow 0\]
of smooth representations of $\RG_n$. Applying the functor
\[
  \mathrm{Hom}_{\RG_n}(\bullet\,,\,|\det|_k^{r+1-n}\otimes\pi^\vee)
 \]
to the above short exact sequence, we get an exact sequence
\begin{equation*}\begin{split}
&0\rightarrow
\mathrm{Hom}_{\RG_n}(\rho,|\det|_\rk^{r+1-n}\otimes\pi^\vee)\rightarrow
 \mathrm{Hom}_{\RG_n}(\omega\otimes \rho,|\det|_\rk^{r+1-n}\otimes\pi^\vee)\\
 &\rightarrow  \mathrm{Hom}_{\RG_n}(\omega^\circ\otimes \rho,|\det|_\rk^{r+1-n}\otimes\pi^\vee)
 \rightarrow  \mathrm{Ext}^{1}_{\RG_n}(\rho,|\det|_\rk^{r+1-n}\otimes\pi^\vee).
\end{split}\end{equation*}
It follows from the assumption on $\tau$ that (\cite[Section
2.2]{B})
\begin{equation*}\begin{split}
& \mathrm{Hom}_{\RG_n}(\rho,|\det|_\rk^{r+1-n}\otimes \pi^\vee)\\
=\ &\mathrm{Hom}_{\RG_n}(\mathrm{ind}_{\RP_{r,n-r}}^{\RG_n}
\sigma\otimes \tau,|\det|_\rk^{r+1-n}\otimes \pi^\vee)=0,\end{split}\end{equation*}
 and that
\begin{equation*}\begin{split}
& \mathrm{Ext}^{1}_{\RG_n}(\rho,|\det|_\rk^{r+1-n}\otimes \pi^\vee)\\
=\ &\mathrm{Ext}^{1}_{\RG_n}(\mathrm{ind}_{\RP_{r,n-r}}^{\RG_n}
\sigma\otimes \tau,|\det|_\rk^{r+1-n}\otimes \pi^\vee)=0.\end{split}\end{equation*}
Therefore, one has that
\begin{equation}\begin{split}\label{iso1} & \mathrm{Hom}_{\RG_n}(\omega\otimes \rho,|\det|_\rk^{r+1-n}\otimes \pi^\vee)\\
=\ & \mathrm{Hom}_{\RG_n}(\omega^\circ\otimes
\rho,|\det|_\rk^{r+1-n}\otimes \pi^\vee).
\end{split}\end{equation}

Recall that  the space $\rk^n$ of column vectors carries the
standard action of $\RG_n$, and has the standard basis $\mathrm e_1,
\mathrm e_2,\cdots,\mathrm e_n$. Denote by $\RP^\circ_i$ the
subgroup of $\RG_n$ fixing $\mathrm e_i$ ($i=1,2,\cdots,n$). Then we
have that (cf. \cite[Exercise 6 (ii)]{H})
\begin{equation}\label{iso2} \omega^\circ\otimes \rho
 \cong (\mathrm{ind}_{\RP^\circ_{r+1}}^{\RG_n}\mathbb
C)\otimes \rho\cong
\mathrm{ind}_{\RP^\circ_{r+1}}^{\RG_n}(\rho|_{\RP^\circ_{r+1}}).
\end{equation}

Now we analyze the representation $\rho|_{\RP^\circ_{r+1}}$. Denote
by $w\in \RG_n$ the permutation matrix which exchanges $\mathrm e_1$
and $\mathrm e_{r+1}$, and fixes $\mathrm e_i$ whenever $i\ne 1,
r+1$. Write $[1_n], \,[w]\in \RP_{r,n-r}\backslash\RG_n$ for the
cosets which are respectively represented by $1_n$ and $w$. Note
that under right translation, $\RP^\circ_{r+1}$ has two obits in
$\RP_{r,n-r}\backslash\RG_n$: one is open and contains $[1_n]$; the
other is closed and contains $[w]$. Set
\[
\RT_i:=\RP_{r,n-r}\cap \RP_i^\circ,\qquad i=1,2,\cdots,n.
\]
Then  the stabilizer of $[1_n]\in \RP_{r,n-r}\backslash\RG_n$ in
$\RP^\circ_{r+1}$ is $\RT_{r+1}$; and the stabilizer of $[w]\in
\RP_{r,n-r}\backslash\RG_n$ in $\RP^\circ_{r+1}$ is $\RT_{1}^w$.
Here and henceforth, for every subgroup $H$ of $\RG_n$, we set
$H^{w}:=wHw$. For every smooth representation $\varrho$ of $H$, we
denote by $\varrho^{w}$ the representation of $H^{w}$ which has the
same underlying space as that of $\varrho$, and has the action
$$\varrho^{w}(h)=\varrho(whw),\qquad h\in H^{w}.$$

Using Mackey theory, we get a short exact sequence
\begin{equation*}
0\rightarrow
\mathrm{ind}_{\RT_{r+1}}^{\RP^\circ_{r+1}}(\sigma\otimes
\tau)|_{\RT_{r+1}} \rightarrow \rho|_{\RP^\circ_{r+1}}\rightarrow
\mathrm{ind}_{\RT_{1}^w}^{\RP^\circ_{r+1}}(\sigma\otimes
\tau)^w|_{\RT_{1}^w}\rightarrow 0\end{equation*} of smooth
representations of  $\RP^\circ_{r+1}$. Inducing this short exact
sequence to $\RG_n$, and using induction in stages, we get  a short
exact sequence
\begin{equation}\label{exact}
0\rightarrow \mathrm{ind}^{\RG_n}_{\RT_{r+1}}(\sigma\otimes
\tau)|_{\RT_{r+1}} \rightarrow
\mathrm{ind}_{\RP^\circ_{r+1}}^{\RG_n}\rho|_{\RP^\circ_{r+1}}
\rightarrow \mathrm{ind}^{\RG_n}_{\RT_{1}^w}(\sigma\otimes
\tau)^w|_{\RT_{1}^w}\rightarrow 0\end{equation} of smooth
representations of  $\RG_n$.

 Recall the mirabolic subgroup $\RP_m$ of $\RG_m$ as in \eqref{mirop}. We view the
representation $(\sigma\otimes \mathrm{ind}_{\RP_r}^{\RG_r}\mathbb
C)\otimes \tau$  of $\RG_r\times \RG_{n-r}$ as a representation of
$\RP_{r,n-r}$ by inflation. Then one has the following isomorphisms
of  representations of $\RP_{r,n-r}^w$:
\begin{equation*}\begin{split}
 &\ \mathrm{ind}_{\RT_{1}^w}^{\RP_{r,n-r}^w}(\sigma\otimes \tau)^w|_{\RT_{1}^w}\\
\cong &\ (\sigma\otimes \tau)^w\otimes
\mathrm{ind}_{\RT_{1}^w}^{\RP_{r,n-r}^w}\mathbb C \qquad (\text{\cite[Exercise 6 (ii)]{H}}) \\
\cong &\ (\sigma\otimes \tau)^w\otimes
(\mathrm{ind}_{\RT_{1}}^{\RP_{r,n-r}}\mathbb C)^w\\
\cong &\ ((\sigma\otimes \mathrm{ind}_{\RP_r}^{\RG_r}\mathbb
C)\otimes \tau)^w. \qquad (\RP_r\backslash \RG_r\cong
\RT_{1}\backslash\RP_{r,n-r})
\end{split}\end{equation*}
This  implies
 \begin{equation*}\begin{split}
 &\ \mathrm{ind}^{\RG_n}_{\RT_{1}^w}(\sigma\otimes \tau)^w|_{\RT_{1}^w}\\
 \cong & \
\mathrm{ind}^{\RG_n}_{\RP_{r,n-r}^w}\mathrm{ind}^{\RP_{r,n-r}^w}_{\RT_{1}^w}(\sigma\otimes
\tau)^w|_{\RT_{1}^w}\\
\cong&\ \mathrm{ind}^{\RG_n}_{\RP_{r,n-r}^w}((\sigma\otimes
\mathrm{ind}_{\RP_r}^{\RG_r}\mathbb C)\otimes \tau)^w\\
\cong&\ \mathrm{ind}^{\RG_n}_{\RP_{r,n-r}}((\sigma\otimes
\mathrm{ind}_{\RP_r}^{\RG_r}\mathbb C)\otimes \tau).
\end{split}\end{equation*}
Then our assumption on $\tau$ implies that (\cite[Section 2.2]{B})
\begin{equation*}
\mathrm{Hom}_{\RG_n}(\mathrm{ind}^{\RG_n}_{\RT_{1}^w}(\sigma\otimes
\tau)^w|_{\RT_{1}^w},|\det|^{r+1-n}_{\rk}\otimes \pi^\vee)=0
\end{equation*}
and
\[
  \mathrm{Ext}^1_{\RG_n}(\mathrm{ind}^{\RG_n}_{\RT_{1}^w}(\sigma\otimes
\tau)^w|_{\RT_{1}^w},|\det|^{r+1-n}_{\rk}\otimes \pi^\vee)=0.
\]
Therefore, applying the functor $\mathrm{Hom}_{\RG_n}(\bullet\, ,\,
|\det|_k^{r+1-n}\pi^\vee)$ to the short exact sequence
\eqref{exact}, we get
\begin{equation*}\begin{split}
&\mathrm{Hom}_{\RG_n}(\mathrm{ind}_{\RP^\circ_{r+1}}^{\RG_n}\rho|_{\RP^\circ_{r+1}},|\det|_\rk^{r+1-n}\otimes \pi^\vee)\\
=\
&\mathrm{Hom}_{\RG_n}(\mathrm{ind}^{\RG_n}_{\RT_{r+1}}(\sigma\otimes \tau)|_{\RT_{r+1}},|\det|_\rk^{r+1-n}\otimes \pi^\vee).
\end{split}\end{equation*}
Combine this  with \eqref{iso1} and \eqref{iso2}, we have that
\begin{equation}\label{iso3}\begin{split}
&\mathrm{Hom}_{\RG_n}(\omega\otimes \rho,|\det|_\rk^{r+1-n}\otimes \pi^\vee)\\
=\
&\mathrm{Hom}_{\RG_n}(\mathrm{ind}^{\RG_n}_{\RT_{r+1}}(\sigma\otimes \tau)|_{\RT_{r+1}},|\det|_\rk^{r+1-n}\otimes \pi^\vee).
\end{split}\end{equation}

Recall that $\tau$ is a supercuspidal representation  of
$\RG_{n-r}$. By the well-known result of Gelfand-Kazhdan
(\cite[Section 5.18]{BZ}), one has that \be\label{sup}
\tau|_{\RP_{n-r}}\cong
\mathrm{ind}_{\RN_{n-r}}^{\RP_{n-r}}\psi^{-1}.\ee Note that
$\RT_{r+1}=(\RG_r\times \RP_{n-r})\ltimes \RU_r$ ($\RU_r$ is defined
as in \eqref{urdef}). Then we have
\begin{equation*}\begin{split}
&\ \mathrm{ind}^{\RG_n}_{\RT_{r+1}}(\sigma\otimes \tau)|_{\RT_{r+1}} \\
= &\ \mathrm{ind}^{\RG_n}_{\RT_{r+1}}(\sigma\otimes
\tau|_{\RP_{n-r}})\\
\cong &\
\mathrm{ind}^{\RG_n}_{\RT_{r+1}}(\sigma\otimes\mathrm{ind}_{\RN_{n-r}}^{\RP_{n-r}}\psi^{-1})
\qquad\ (\text{by }\eqref{sup})\\
\cong &\
\mathrm{ind}^{\RG_n}_{\RT_{r+1}}\mathrm{ind}_{\RR_r}^{\RT_{r+1}}(\sigma\otimes\psi^{-1})\
\qquad\,  (\RR_r\backslash\RT_{r+1}
\cong \RN_{n-r}\backslash\RP_{n-r})\\
\cong &\ \mathrm{ind}^{\RG_n}_{\RR_r}(\sigma\otimes \psi^{-1})\\
\cong &\ \mathrm{ind}^{\RG_n}_{\RR_r}(\sigma\otimes
\chi^{-1}).\qquad \qquad\qquad \ (\chi=1\otimes \psi)
\end{split}\end{equation*}
Together with \eqref{iso3}, this implies
$$ \mathrm{Hom}_{\RG_n}(\omega\otimes \rho,|\det|_\rk^{r+1-n}\otimes \pi^\vee)=
 \mathrm{Hom}_{\RG_n}(\mathrm{ind}^{\RG_n}_{\RR_r}(\sigma\otimes \chi^{-1}),|\det|_\rk^{r+1-n}\otimes \pi^\vee).$$
 Finally, by Frobenius reciprocity (\cite[Proposition 2.29]{BZ}), one concludes from the above equality that
 $$\mathrm{Hom}_{\RG_n}(\omega\otimes \rho,|\det|_\rk^{r+1-n}\otimes \pi^\vee)=
 \mathrm{Hom}_{\RR_r}(\pi,\sigma^{\vee}\otimes \chi).$$
This finishes the proof of the claim.


\begin{thebibliography}{BGK}
\bibitem[AG]{AG} A. Aizenbud and D. Gourevitch, \textit{Multiplicity one theorem for $(\mathrm{GL}_{n+1}(\mathbb R),
\mathrm{GL}_n(\mathbb R))$}, Sel. Math. New Ser. 15 (2009), 271-294.

\bibitem[AGRS]{AGRS} A. Aizenbud, D. Gourevitch, S. Rallis and G. Schiffmann, \textit{Multiplicity one theorems},
Ann. of Math. (2) 172 (2010), no.2, 1407-1434.

\bibitem[Be]{B} J. Bernstein, \textit{Draft of: representations of p-adic groups}, Lectures by Joseph Bernstein, Harvard University Fall 1992, written by Karl E. Rumelhart.

\bibitem[BK]{BK}
J. Bernstein and B. Kr\"{o}tz, \textit{Smooth Fr\'echet globalizations
of Harish-Chandra modules}, arXiv:0812.1684.

\bibitem[BZ]{BZ} I.N. Bernstein and A.V. Zelevinsky, \textit{Representations of $\GL(n, F)$ where $F$ is
a non-archimedean local field}, Uspekhi Mat. Nauk. 31: 3 (1976),
5-70.

\bibitem[Ca]{C} W. Casselman, \textit{Canonical extensions of Harish-Charadra modules to represnetation of $G$},
Canad. J. Math. 41 (1989), no.3, 385-438.

\bibitem[CHM]{CHM} W. Casselman, H. Hecht and Milicic, \textit{Bruhat filtrations and Whittaker vectors for
real groups}, the Mathematial Legacy of Harish-Chandra, Proc. Symp.
Pure Math. 68, Amer. Math. Soc. Providence, RI (2000), 151-190.

\bibitem[CPS]{CPS} J. Cogdell and I.I. Piatetski-Shapiro, \textit{Remarks on Rankin-Selberg Convolutions}, Contributions to Automorphic Forms, Geometry, and Number Theory. The Johns Hopkins University Press, 2004.

\bibitem[dC]{du}
F. du Cloux, \textit{Sur les reprsentations diffrentiables des
groupes de Lie algbriques}, Ann. Sci. Ecole Norm. Sup. 24 (1991),
no. 3, 257-318.

%\bibitem[CP1]{CP1} J.W. Cogdel, I.I. Piatetski-Shapiro, Converse Theorem for $\GL_n$, Publ. Math. IHES 79 (1994), 157-214.
%\bibitem[CP2]{CP2} J.W. Cogdel, I.I. Piatetski-Shapiro, Converse Theorem for $\GL_n$, II, J. reine angew. Math. 507 (1999), 165-188.

\bibitem[GGP]{GGP} W.T. Gan, B.H. Gross and D. Prasad, \textit{Symplectic local root numbers, central critical L-values,
and restriction problems in the representation theory of classical
groups}, Ast\'{e}risque 346 (2012), 117-170.

\bibitem[H]{H} F. Herzig, \textit{Smooth representations of p-adic reductive groups,} http: //www.math.toronto.edu/~herzig/smooth representations.pdf

\bibitem[Ja]{J} H. Jacquet,  \textit{Archimedean Rankin-Selberg integrals}, in ``Automorphic Forms and
L-functions II: Local Aspects", Proceedings of a Workshop in Honor
of Steve Gelbart on the Occasion of his 60th Birthday, Contemporary
Mathematics 488-9, AMS and BIU (2009) 57-172.

\bibitem[JS1]{JS1} H. Jacquet and J.A. Shalika, \textit{On Euler products and the classification of automorphic
representations I}, Amer. J. of Math. 103 (1981) no.3, 199-558.

\bibitem[JS2]{JS2} H. Jacquet and J.A. Shalika, \textit{Rankin-Selberg covolutions: Archimedean theory}, in Fetschrift
in honor of I.I. Piatetski-Shapiro, Part I., Weizmann Science Press,
Jerusalem, (1990) 125-207.

%\bibitem[JPS1]{JPS1} H. Jacquet, I.I. Piatetski-Shapiro and J. A. Shalika, \textit{Automorphic forms on $\GL_3$, I \& II}, Ann. Math. 109 (1979) 169-258.

\bibitem[JPS]{JPS} H. Jacquet, I.I. Piatetski-Shapiro and J.A. Shalika,  \textit{Rankin-Selberg convolutions}, Amer. J. of Math. 105 (1983), 777-815.

\bibitem[JSZ]{JSZ} D. Jiang, B. Sun and C.-B. Zhu, \textit{Uniqueness of Bessel models: the Archimedean case},
Geom. Funct. Anal. 20 (2010), no.3, 690-709.

\bibitem[LS]{LS} Y. Liu and B. Sun, \textit{Uniqueness of Fourier-Jacobi Models: the archimedean case}, J. Funct. Anal. 265 (2013), 3325-3344.


\bibitem[Lo]{Lo} L. Loomis, \textit{An Introduction to Abstract Harmonic Analysis},
D. Van Nostrand Co., New York, 1953.

\bibitem[Sa]{Sa} F. Sauvegeot, \textit{Principle densit\'{e} pur les groupes r\'{e}ductifs}, Compos. Math.  108 (1997), no.2,
151-184.

\bibitem[Sha]{Sh} J.A. Shalika, \textit{The multiplicity one theorem for $\mathrm{GL}_n$}, Ann. of Math. (2) 100 (1974), 171-193.

\bibitem[Shi]{Shi}
M. Shiota, \textit{Nash Manifolds}, Lect. Notes Math., vol. 1269,
Springer-Verlag, 1987.

\bibitem[SV]{SV} B. Speh and D.A. Vogan, \textit{Reducibility of generalized principal series representatitions},
Acta Math. 145 (1980), no.3-4, 227-299.

\bibitem[Su1]{S} B. Sun, \textit{Multiplicity one theorems for Fourier-Jacobi models}, American J. Math.134 (2012), 1655-1678.

\bibitem[Su2]{Sun2} B. Sun, \textit{Almost linear Nash groups}, arXiv:1310.8011

\bibitem[SZ]{SZ} B. Sun and C.-B. Zhu, \textit{Multiplicity one theorems: the Archimedean case}, Ann. of Math. (2) 175 (2012), 23-44.

\bibitem[Tr]{Tr}
F. Treves, \textit{Topological Vector Spaces, Distributions and
Kernels}, Academic Press, New York, 1967.

\bibitem[Wa]{W} N. Wallach, \textit{Real Reductive Groups I \& II}, Academic Press, Pure and Applied Mathematis, vols. 132 \& 132 II,
1988 \& 1992.

\end{thebibliography}
\end{document}